\documentclass[12pt]{amsart}%
\usepackage{amscd,amsmath,latexsym,amsthm,amsfonts,amssymb,graphicx,color,geometry}
\usepackage{enumerate}
\usepackage{hyperref}
\usepackage[norefs,nocites]{refcheck}
\usepackage{amsmath}
\usepackage{amsfonts}
\usepackage{amssymb}
\usepackage{graphicx}%
\providecommand{\U}[1]{\protect\rule{.1in}{.1in}}
\hypersetup{
	colorlinks=true,
	linkcolor=blue,
	citecolor=cyan
		}
\allowdisplaybreaks
\newtheorem{theorem}{Theorem}[section]

\numberwithin{equation}{section}

\begin{document}
\title[Macphail's Theorem revisited]{Macphail's Theorem revisited}
\author[D. Pellegrino]{Daniel Pellegrino}
\address{Departamento de Matem\'{a}tica \\
Universidade Federal da Para\'{\i}ba \\
58.051-900 - Jo\~{a}o Pessoa, Brazil.}
\email{daniel.pellegrino@academico.ufpb.br and dmpellegrino@gmail.com}
\author[J.. Silva]{Janiely Silva}
\email{janiely.silva@estudantes.ufpb.br}
\thanks{D. Pellegrino is supported by CNPq Grant 307327/2017-5 and Grant 2019/0014
Paraiba State Research Foundation (FAPESQ) and J. Silva is supported by CAPES}

\begin{abstract}
In 1947, M. S. Macphail constructed a series in $\ell_{1}$ that converges
unconditionally but does not converge absolutely. According to the literature,
this result helped Dvoretzky and Rogers to finally answer a long standing
problem of Banach Space Theory, by showing that in all infinite-dimensional
Banach spaces, there exists an unconditionally summable sequence that fails to
be absolutely summable. More precisely, the Dvoretzky--Rogers Theorem asserts
that in every infinite-dimensional Banach space $E$ there exists an
unconditionally convergent series ${\textstyle\sum}x^{(j)}$ such that
${\textstyle\sum}\Vert x^{(j)}\Vert^{^{2-\varepsilon}}=\infty$ for all
$\varepsilon>0.$ Their proof is non-constructive and Macphail%
%TCIMACRO{\U{b4}}%
%BeginExpansion
's result for $E=\ell_{1}$ provides a constructive proof just for
$\varepsilon\geq1.$ In this note we revisit Machphail's paper and present two
alternative constructions that work for all $\varepsilon>0.$

\end{abstract}
\maketitle

\section{Introduction}

A series in a Banach space is said to be unconditionally convergent if all
rearrangements of that series converge (to the same vector). An old result due
to Dirichlet (1829) asserts that a series $%
%TCIMACRO{\tsum }%
%BeginExpansion
{\textstyle\sum}
%EndExpansion
x^{(j)}$ of real or complex scalars is unconditionally convergent precisely
when it is absolutely convergent, i.e., when $%
%TCIMACRO{\tsum }%
%BeginExpansion
{\textstyle\sum}
%EndExpansion
\Vert x^{(j)}\Vert$ converges. The same characterization holds for
finite-dimensional Banach spaces, see, for example, \cite[Theorem 1.3.5]{kk}.
In the infinite-dimensional framework it seems that unconditionally convergent
series were first investigated by Orlicz about 100 years after Dirichlet. The
subject attracted the attention of Banach, Mazur and others, who proposed the
following problem (see \cite[Problem 122]{mau} and also \cite{pie} for further
details): in an arbitrary infinite-dimensional Banach space, does there exist
an unconditionally convergent series that does not converge absolutely?

In most of the classical Banach spaces, examples of unconditionally convergent
series that do not converge absolutely were easily obtained. For instance, the
sequence $(j^{-1}e_{j})_{j=1}^{\infty}$ is unconditionally summable in
$\ell_{2}$ but it fails to be absolutely summable. The same example works for
all $\ell_{p}$ spaces for any $p>1,$ but it is useless for $\ell_{1}$. In
1947, Macphail (\cite{mac}) proved that the answer was also positive in
$\ell_{1}$. Macphail's approach has shown that $\ell_{1}$ is essentially the
critical case for all infinite-dimensional Banach spaces; this becomes rather
clear, for instance, in the approach described in \cite[\S 2 and Lemma 1]{rut}
and, according to \cite[pages 2 and 20]{diestel}, inspired Dvoretzky and
Rogers in 1950 (\cite{dr}) to transplant in a highly nontrivial fashion the
construction of Macphail to any infinite-dimensional Banach space and obtain a
definitive solution to the problem.

For a Banach space $E$ and $S$ a finite sequence of vectors $x_{1}%
,\ldots,x_{n}$ in $E$, Macphail defined
\[
G(S)=\frac{\sup_{\sigma}\left\Vert \sum_{i\in\sigma}x_{i}\right\Vert%
}{\sum_{i=1}^{n}\Vert x_{i}\Vert},
\]
where the supremum runs over all subset $\sigma\subseteq\{1,\ldots,n\}$, and
\[
\mu(E)=\inf G(S),
\]
where the infimum is considered over all finite sequences $S\subset E$. The
parameter $\mu(E)$ is now known as Macphail's constant of $E$. In his paper,
it is first observed that if $\mu(E)=0$, then unconditional convergence does
not imply absolute convergence in $E.$ The main result of \cite{mac} shows
that $\mu(\ell_{1})=0$.

Some alternative proofs of Macphail's result and of the Dvoretzky--Rogers
Theorem appeared later, but they were not constructive (see, for instance,
\cite[page 145]{Kvaratskhelia} and \cite{ara, israel} and the references
therein). We also refer the interested reader to \cite{bayart} for a
non-constructive approach in $\ell_{1}$ using the Kahane--Salem--Zygmund
inequality (see also \cite{mas} for details on the Kahane--Salem--Zygmund
inequality) and \cite{gordon, rut1} for the asymptotic behavior of $\mu
(\ell_{p}^{n})$.

The striking result of Dvoretzky--Rogers, showing that $\mu(E)=0$ for all
infinite-dimensional Banach spaces, attracted the attention of Grothendieck
who referred to the Dvoretzky--Rogers Theorem as the unique decisive result in
the finner metric theory of general Banach spaces known at that time (see
\cite[page 20]{diestel}). Additionally, the result of Dvoretzky and Rogers
answers much more than what is asked in the original problem of Banach's
school. In more precise terms, if $E$ is an infinite-dimensional Banach space,
the Dvoretzky--Rogers Theorem assures the existence of an unconditionally
convergent series $%
%TCIMACRO{\tsum }%
%BeginExpansion
{\textstyle\sum}
%EndExpansion
x^{(j)}$ in $X$ such that%
\begin{equation}%
%TCIMACRO{\tsum \limits_{j=1}^{\infty}}%
%BeginExpansion
{\textstyle\sum\limits_{j=1}^{\infty}}
%EndExpansion
\left\Vert x^{(j)}\right\Vert ^{2-\varepsilon}=\infty\label{qqqq}%
\end{equation}
for all $\varepsilon>0.$ However, the proof of the Dvoretzky--Rogers Theorem
do not offer an explicit construction of such sequence. In this note we
revisit Macphail's paper and provide two alternative constructions of such a
sequence in $\ell_{p}$, for all $p\in\lbrack1,2]$ (the case $p>2$ is obvious
and for this reason we do not consider it); unlike Macphail's approach which
is valid for $\varepsilon\geq1,$ the construction presented here completely
satisfies the above statement, i.e., it provides an unconditionally convergent
series $%
%TCIMACRO{\tsum }%
%BeginExpansion
{\textstyle\sum}
%EndExpansion
x^{(j)}$ in $\ell_{p}$ satisfying (\ref{qqqq}) for all $\varepsilon>0$.
Similar constructions may be known for experts in the field but we were not
able to find in the literature.

\section{The fisrt approach}

The first alternative construction is valid just for complex scalars and rests
in the properties of some special matrices that date back to the works of
Toeplitz \cite{toe}: 

\begin{theorem}
\label{casecomp} Let $p\in\lbrack1,2]$ and $\alpha>1$ a positive integer. For
all positive integers $j$, let
\[
A_{j}:=\left[  1/2+\sqrt{1/4+\log_{\alpha}\left(  \left(  j+1\right)
/2\right)  },1/2+\sqrt{1/4+\log_{\alpha}j}\right]  \cap\mathbb{N}.
\]
The sequence $(x^{(j)})_{j=1}^{\infty}$ defined by
\[
x^{(j)}={\alpha}^{{\frac{(1-k_{j})(k_{j}(2+p)+2p)}{2p}}}\cdot\sum
_{s=1}^{\alpha^{k_{j}(k_{j}-1)}}\left(  \exp\left(  {\left[  j\alpha
^{k_{j}\left(  1-k_{j}\right)  }+\alpha^{k_{j}\left(  1-k_{j}\right)
}-1\right]  2\pi si}\right)  \cdot e_{\alpha^{k_{j}(k_{j}-1)}+s-1}\right)
\]
if $A_{j}\neq\emptyset$, where $k_{j}$ is the unique element of $A_{j}$, and
$x^{(j)}=0$ otherwise, is unconditionally summable in $\ell_{p}$ (over the
complex scalar field) and
\[%
%TCIMACRO{\tsum \limits_{j=1}^{\infty}}%
%BeginExpansion
{\textstyle\sum\limits_{j=1}^{\infty}}
%EndExpansion
\left\Vert x^{(j)}\right\Vert ^{2-\varepsilon}=\infty
\]
for all $\varepsilon>0.$
\end{theorem}

The statement may seem somewhat complicated at first glance, but the
construction is fairly simple. We shall first observe that the interval
\[
\left[  1/2+\sqrt{1/4+\log_{\alpha}\left(  \left(  j+1\right)  /2\right)
},\text{ }1/2+\sqrt{1/4+\log_{\alpha}j}\right]
\]
has size smaller than $1$ and thus $A_{j}$ is, in fact, either empty or formed
by a single integer. A straightforward computation shows that when
$j\in\left\{  \alpha^{k\left(  k-1\right)  },\alpha^{k\left(  k-1\right)
}+1,\ldots,2\alpha^{k\left(  k-1\right)  }-1\right\}  $ for a certain positive
integer $k$, we have $A_{j}=\{k\}$; otherwise $A_{j}$ is empty.

\begin{proof}
For each positive integer $n$, consider the $n\times n$ matrix $(a_{rs}%
^{(n)})$ with $a_{rs}^{(n)}=\exp(2\pi irs/n);$ it is obvious that
$|a_{rs}^{(n)}|=1$ and one can also check that
\begin{equation}
\sum_{s=1}^{n}a_{rs}^{(n)}\overline{a_{ts}^{(n)}}=n\delta_{rt},\label{i9x}%
\end{equation}
where $\delta_{rt}$ denotes the Kronecker delta. In order to prove (\ref{i9x})
note that
\[
\sum_{s=1}^{n}a_{rs}^{(n)}\overline{a_{ts}^{(n)}}=\sum_{s=1}^{n}\exp\left(
2\pi i\frac{rs}{n}\right)  \cdot\exp\left(  -2\pi i\frac{ts}{n}\right)
=\sum_{s=1}^{n}\exp\left(  2\pi i\frac{(r-t)s}{n}\right)  .
\]
If $r=t$, it is obvious that $\sum_{s=1}^{n}a_{rs}^{(n)}\overline{a_{ts}%
^{(n)}}=n$. On the other hand, if $r\neq t$, we have
\[
\sum_{s=1}^{n}\exp\left(  2\pi i\frac{(r-t)s}{n}\right)  =\frac{\exp\left(
2\pi i\frac{(r-t)}{n}\right)  -\exp\left(  2\pi i\frac{(r-t)(n+1)}{n}\right)
}{1-\exp\left(  2\pi i\frac{(r-t)}{n}\right)  }%
\]
and, recalling that $\exp(\pi i)=-1,$ a straightforward computation gives us
\[
\sum_{s=1}^{n}a_{rs}^{(n)}\overline{a_{ts}^{(n)}}=0.
\]
Let $p^{\ast}$ be the conjugate of $p$, i.e.,%
\[
\frac{1}{p}+\frac{1}{p^{\ast}}=1
\]
and let $\ell_{p}^{n}$ denote $\mathbb{C}^{n}$ with the $\ell_{p}$-norm. For
unit vectors $y^{(1)}\in\ell_{p^{\ast}}^{n}$ and $y^{(2)}\in\ell_{\infty}^{n}%
$, we have%
\begin{equation}
\left\vert \sum_{i_{1},i_{2}=1}^{n}a_{i_{1}i_{2}}^{(n)}y_{i_{1}}^{(1)}%
y_{i_{2}}^{(2)}\right\vert \leq n^{\frac{1}{2}+\frac{1}{p}}.\label{0987}%
\end{equation}
The proof of (\ref{0987}) follows by Schur's test (also known as Young's
inequality). For the sake of completeness we present a proof here following
the lines of \cite[pages 30 and 31]{djd}. By the H\"{o}lder inequality, we
have
\begin{align*}
\left\vert \sum_{i_{1},i_{2}=1}^{n}a_{i_{1}i_{2}}^{(n)}y_{i_{1}}^{(1)}%
y_{i_{2}}^{(2)}\right\vert  &  \leq\sum_{i_{2}=1}^{n}\left\vert \sum_{i_{1}%
=1}^{n}a_{i_{1}i_{2}}^{(n)}y_{i_{1}}^{(1)}\right\vert |y_{i_{2}}^{(2)}|\\
&  \leq\left(  \sum_{i_{2}=1}^{n}|y_{i_{2}}^{(2)}|^{2}\right)  ^{\frac{1}{2}%
}\left(  \sum_{i_{2}=1}^{n}\left\vert \sum_{i_{1}=1}^{n}a_{i_{1}i_{2}}%
^{(n)}y_{i_{1}}^{(1)}\right\vert ^{2}\right)  ^{\frac{1}{2}}\\
&  \leq n^{\frac{1}{2}}\left(  \sum_{i_{2}=1}^{n}\left\vert \sum_{i_{1}=1}%
^{n}a_{i_{1}i_{2}}^{(n)}y_{i_{1}}^{(1)}\right\vert ^{2}\right)  ^{\frac{1}{2}%
}.
\end{align*}
Since%
\[
\left(  \sum_{i_{2}=1}^{n}\left\vert \sum_{i_{1}=1}^{n}a_{i_{1}i_{2}}%
^{(n)}y_{i_{1}}^{(1)}\right\vert ^{2}\right)  ^{\frac{1}{2}}=\left(
\sum_{i_{2}=1}^{n}\sum_{\substack{i_{1}=1\\j_{1}=1}}^{n}y_{i_{1}}%
^{(1)}\overline{y_{j_{1}}^{(1)}}a_{i_{1}i_{2}}^{(n)}\overline{a_{j_{1}i_{2}%
}^{(n)}}\right)  ^{\frac{1}{2}}=\left(  \sum_{\substack{i_{1}=1\\j_{1}=1}%
}^{n}y_{i_{1}}^{(1)}\overline{y_{j_{1}}^{(1)}}\sum_{i_{2}=1}^{n}a_{i_{1}i_{2}%
}^{(n)}\overline{a_{j_{1}i_{2}}^{(n)}}\right)  ^{\frac{1}{2}},
\]
by (\ref{i9x}) and by the H\"{o}lder inequality, we have
\begin{align*}
\left\vert \sum_{i_{1},i_{2}=1}^{n}a_{i_{1}i_{2}}^{(n)}y_{i_{1}}^{(1)}%
y_{i_{2}}^{(2)}\right\vert  &  \leq n^{\frac{1}{2}}\left(  \sum
_{\substack{i_{1}=1\\j_{1}=1}}^{n}y_{i_{1}}^{(1)}\overline{y_{j_{1}}^{(1)}%
}n\delta_{i_{1}j_{1}}\right)  ^{\frac{1}{2}}\\
&  =n\left(  \sum_{i_{1}=1}^{n}|y_{i_{1}}^{(1)}|^{2}\right)  ^{\frac{1}{2}}\\
&  \leq n\left(  \sum_{i_{1}=1}^{n}1\right)  ^{\frac{1}{2}-\frac{1}{p^{\ast}}%
}\left(  \sum_{i_{1}=1}^{n}|y_{i_{1}}^{(1)}|^{p^{\ast}}\right)  ^{\frac
{1}{p^{\ast}}}\\
&  =n^{\frac{1}{2}+\frac{1}{p}}.
\end{align*}
Thus, by (\ref{0987}), we obtain
\begin{equation} \label{rr}
\sup_{\left\Vert \varphi\right\Vert _{p^{\ast}}\leq1}\sum\limits_{r=1}%
^{n}\left\vert \sum\limits_{s=1}^{n}\varphi_{s}a_{rs}^{(n)}\right\vert
=\sup_{\left\Vert \varphi\right\Vert _{p^{\ast}}\leq1}\sup_{\left\Vert
	\psi\right\Vert _{\infty}\leq1}\left\vert \sum\limits_{r,s=1}^{n}\psi
_{r}\varphi_{s}a_{rs}^{(n)}\right\vert \leq n^{\frac{1}{2}+\frac{1}{p}},
\end{equation}
where $\varphi=(\varphi_{1},\ldots,\varphi_{n})$ and $\psi=(\psi_{1}%
,\ldots,\psi_{n})$.

Let $j_{k}:=\alpha^{k(k-1)}$ for all positive integers $k$. Define the
sequence $(x^{(j)})_{j=1}^{\infty}$ as follows: if $j\in\{j_{k},\ldots
,2j_{k}-1\}$, for a certain positive integer $k$, consider
\[
x^{(j)}=j_{k}^{-\left(  \frac{1}{2}+\frac{1}{p}+\frac{1}{k}\right)  }\left(
\sum_{s=1}^{j_{k}}a_{rs}^{(j_{k})}e_{j_{k}+s-1}\right)  ,
\]
where $r=j-j_{k}+1$; otherwise take $x^{(j)}=0$. For all $\varphi\in(\ell
_{p})^{\ast}$ it is plain that
\begin{equation}\label{lim}
\sup_{\left\Vert \varphi\right\Vert _{p^{\ast}}\leq1}%
%TCIMACRO{\tsum \limits_{j=1}^{\infty}}%
%BeginExpansion
{\textstyle\sum\limits_{j=1}^{\infty}}
%EndExpansion
\left\vert \varphi\left(  x^{(j)}\right)  \right\vert \leq\sum_{k=1}^{\infty
}\left(  \sup_{\left\Vert \varphi\right\Vert _{p^{{\ast}}}\leq1}\sum_{j=j_{k}%
}^{2j_{k}-1}\left\vert \varphi\left(  x^{(j)}\right)  \right\vert \right)  .
\end{equation}

By the definition of $x^{(j)}:=(x_{m}^{(j)})_{m=1}^{\infty}$, we obtain
\begin{align}
\sum_{j=j_{k}}^{2j_{k}-1}\left\vert \varphi\left(  x^{(j)}\right)
\right\vert  &  =\sum_{j=j_{k}}^{2j_{k}-1}\left\vert \sum_{m=1}^{\infty
}\varphi_{m}x_{m}^{(j)}\right\vert \label{88}\\
&  =\sum_{j=j_{k}}^{2j_{k}-1}\left\vert j_{k}^{-\left(  \frac{1}{2}+\frac
{1}{p}+\frac{1}{k}\right)  }\sum_{m=j_{k}}^{2j_{k}-1}\varphi_{m}%
a_{[j-j_{k}+1][m-j_{k}+1]}^{(j_{k})}\right\vert \nonumber
\end{align}
for all positive integers $k$. Note that we can assign new indices to the sums
$\sum_{j=j_{k}}^{2j_{k}-1}$ and $\sum_{m=j_{k}}^{2j_{k}-1}$ as follows:
\begin{align*}
\sum_{j=j_{k}}^{2j_{k}-1}\left\vert j_{k}^{-\left(  \frac{1}{2}+\frac{1}%
{p}+\frac{1}{k}\right)  }\sum_{m=j_{k}}^{2j_{k}-1}\varphi_{m}a_{[j-j_{k}%
+1][m-j_{k}+1]}^{(j_{k})}\right\vert  &  =\sum_{r=1}^{j_{k}}\left\vert
j_{k}^{-\left(  \frac{1}{2}+\frac{1}{p}+\frac{1}{k}\right)  }\sum_{s=1}%
^{j_{k}}\varphi_{(j_{k}+s-1)}a_{rs}^{(j_{k})}\right\vert \\
&  =j_{k}^{-\left(  \frac{1}{2}+\frac{1}{p}+\frac{1}{k}\right)  }\sum
_{r=1}^{j_{k}}\left\vert \sum_{s=1}^{j_{k}}\varphi_{(j_{k}+s-1)}a_{rs}%
^{(j_{k})}\right\vert .
\end{align*}
Hence, for $n\geq2$, by (\ref{88}) and (\ref{rr}), we have
\begin{align}
\sum_{k=n}^{\infty}\left(  \sup_{\left\Vert \varphi\right\Vert _{p^{\ast}}%
\leq1}\sum_{j=j_{k}}^{2j_{k}-1}\left\vert \varphi\left(  x^{(j)}\right)
\right\vert \right)   &  =\sum_{k=n}^{\infty}\left(  \sup_{\left\Vert
\varphi\right\Vert _{p^{\ast}}\leq1}\sum_{j=j_{k}}^{2j_{k}-1}\left\vert
j_{k}^{-\left(  \frac{1}{2}+\frac{1}{p}+\frac{1}{k}\right)  }\sum_{m=j_{k}%
}^{2j_{k}-1}\varphi_{m}a_{[j-j_{k}+1][m-j_{k}+1]}^{(j_{k})}\right\vert
\right)  \label{aac}\\
&  =\sum_{k=n}^{\infty}\left(  \sup_{\left\Vert \varphi\right\Vert _{p^{\ast}%
}\leq1}j_{k}^{-\left(  \frac{1}{2}+\frac{1}{p}+\frac{1}{k}\right)  }\sum
_{r=1}^{j_{k}}\left\vert \sum_{s=1}^{j_{k}}\varphi_{(j_{k}+s-1)}a_{rs}%
^{(j_{k})}\right\vert \right)  \nonumber\\
&  =\sum_{k=n}^{\infty}\left(  j_{k}^{-\left(  \frac{1}{2}+\frac{1}{p}%
+\frac{1}{k}\right)  }\sup_{\left\Vert \varphi\right\Vert _{p^{\ast}}\leq
1}\sum_{r=1}^{j_{k}}\left\vert \sum_{s=1}^{j_{k}}\varphi_{(j_{k}+s-1)}%
a_{rs}^{(j_{k})}\right\vert \right)  \nonumber\\
&  \leq\sum_{k=n}^{\infty}\left(  j_{k}^{-\left(  \frac{1}{2}+\frac{1}%
{p}+\frac{1}{k}\right)  }j_{k}^{\frac{1}{2}+\frac{1}{p}}\right)  \nonumber\\
&  =\sum_{k=n}^{\infty}\alpha^{1-k}<\infty.\nonumber
\end{align}
A well-known necessary and sufficient condition for a series $\sum
_{n=1}^{\infty}y_{n}$ to be unconditionally summable (see \cite[Theorem
1.5]{diestel}) is that, given $\delta>0$, there is $n_{\delta}\in\mathbb{N}$
such that
\[
\left\Vert \sum_{n\in M}y_{n}\right\Vert <\delta
\]
whenever $M$ is a finite subset of $\mathbb{N}$ with $\min M>n_{\delta}$.

By (\ref{lim}) and (\ref{aac}) we have
\[
\lim_{n\rightarrow\infty}\sup_{\left\Vert \varphi\right\Vert _{p^{\ast}}\leq1}%
%TCIMACRO{\tsum \limits_{j=n}^{\infty}}%
%BeginExpansion
{\textstyle\sum\limits_{j=n}^{\infty}}
%EndExpansion
\left\vert \varphi\left(  x^{(j)}\right)  \right\vert =0.
\]
Hence, given $\delta>0$, there is $n_{\delta}\in\mathbb{N}$ such that $n\geq
n_{\delta}$ implies
\[
\sup_{\left\Vert \varphi\right\Vert _{p^{\ast}}\leq1}%
%TCIMACRO{\tsum \limits_{i=n}^{\infty}}%
%BeginExpansion
{\textstyle\sum\limits_{j=n}^{\infty}}
%EndExpansion
\left\vert \varphi\left(  x^{(j)}\right)  \right\vert <\delta.
\]
Then, for all finite sets $M\subset\mathbb{N}$ such that $M\subset\{n_{\delta
},n_{\delta}+1,\ldots\}$, it follows that
\[
\left\Vert \sum_{j\in M}x^{(j)}\right\Vert \leq\sup_{\left\Vert \varphi
\right\Vert _{p^{\ast}}\leq1}\sum_{j\in M}|\varphi(x^{(j)})|\leq
\sup_{\left\Vert \varphi\right\Vert _{p^{\ast}}\leq1}\sum_{j=n}^{\infty
}|\varphi(x^{(j)})|<\delta.
\]
Therefore, $(x^{(j)})_{j=1}^{\infty}$ is unconditionally summable.

On the other hand, when $j\in\{j_{k},\ldots,2j_{k}-1\}$, for a certain
positive integer $k$, we have
\[
\left\Vert x^{(j)}\right\Vert =\left(  j_{k}\cdot\left(  j_{k}^{-\left(
\frac{1}{2}+\frac{1}{p}+\frac{1}{k}\right)  }\right)  ^{p}\right)  ^{\frac
{1}{p}}=j_{k}^{-\left(  \frac{1}{2}+\frac{1}{k}\right)  };
\]
otherwise $\Vert x^{(j)}\Vert=0$. Then, we have
\[%
%TCIMACRO{\tsum }%
%BeginExpansion
{\textstyle\sum}
%EndExpansion
\left\Vert x^{(j)}\right\Vert ^{r}=\sum_{k=1}^{\infty}j_{k}\cdot\left(
j_{k}^{-\left(  \frac{1}{2}+\frac{1}{k}\right)  }\right)  ^{r}=\sum
_{k=1}^{\infty}j_{k}^{1-\frac{r}{2}-\frac{r}{k}}=\sum_{k=1}^{\infty}%
\alpha^{k(k-1)\left(  1-\frac{r}{2}-\frac{r}{k}\right)  },
\]
that diverges for all $r<2$, because $\lim_{k\rightarrow\infty}\alpha
^{k(k-1)\left(  1-\frac{r}{2}-\frac{r}{k}\right)  }\neq0$ whenever $r<2$.

Note that, considering, for all positive integers $j$,
\[
A_{j}:=\left[  1/2+\sqrt{1/4+\log_{\alpha}\left(  \left(  j+1\right)
/2\right)  },1/2+\sqrt{1/4+\log_{\alpha}j}\right]  \cap\mathbb{N},
\]
each $x^{(j)}$ is precisely
\[
x^{(j)}={\alpha}^{{\frac{(1-k_{j})(k_{j}(2+p)+2p)}{2p}}}\cdot\sum
_{s=1}^{\alpha^{k_{j}(k_{j}-1)}}\left(  \exp\left(  {\left[  j\alpha
^{k_{j}\left(  1-k_{j}\right)  }+\alpha^{k_{j}\left(  1-k_{j}\right)
}-1\right]  2\pi si}\right)  \cdot e_{\alpha^{k_{j}(k_{j}-1)}+s-1}\right)
\]
when $A_{j}\neq\emptyset$, where $k_{j}$ is the unique element of $A_{j}$, and
$x^{(j)}=0$ otherwise. This concludes the proof.
\end{proof}

\section{The second approach}

In this section, following the lines of the proof of Theorem \ref{casecomp},
we borrow an argument credited to Pelczynski to adapt the previous
construction to the real scalar field.

The core of the proof rests on properties of the Walsh system (see \cite[page
8]{wal}), which is the set formed by the following functions:%

\begin{align*}
f_{0}(x)  &  =1,~~0\leq x\leq1,\\
f_{1}(x)  &  =%
\begin{cases}
1, & 0\leq x<\frac{1}{2},\\
-1, & \frac{1}{2}<x\leq1,
\end{cases}
\\
f_{2}^{(1)}(x)  &  =%
\begin{cases}
1, & 0\leq x<\frac{1}{4},~~\frac{3}{4}<x\leq1\\
-1, & \frac{1}{4}<x<\frac{3}{4},
\end{cases}
\\
f_{2}^{(2)}(x)  &  =%
\begin{cases}
1, & 0\leq x<\frac{1}{4},~~\frac{1}{2}<x<\frac{3}{4},\\
-1, & \frac{1}{4}<x<\frac{1}{2},~~\frac{3}{4}<x\leq1,
\end{cases}
\\
&  \vdots\\
f_{n+1}^{(2k-1)}(x)  &  =%
\begin{cases}
f_{n}^{(k)}(2x), & 0\leq x<\frac{1}{2},\\
(-1)^{k+1}f_{n}^{(k)}(2x-1), & \frac{1}{2}<x\leq1,
\end{cases}
\\
f_{n+1}^{(2k)}(x)  &  =%
\begin{cases}
f_{n}^{(k)}(2x), & 0\leq x<\frac{1}{2}\\
(-1)^{k}f_{n}^{(k)}(2x-1), & \frac{1}{2}<x\leq1,
\end{cases}
\end{align*}
where $n\in\mathbb{N}$ and $k=1,2,\ldots,2^{n-1}$.

\begin{theorem}
Let $m$ be a positive integer and $g_{1},\ldots,g_{2^{m}}$ the first
$n:=2^{m}$ functions of the Walsh system. For $i,j=1,\ldots,2^{m}$, let
$a_{i,j}^{(n)}$ the value that each $g_{j}$ takes in the interval
$((i-1)/2^{m},i/2^{m})$. Let $p\in\lbrack1,2]$. For all positive integers $j$,
let
\[
A_{j}:=\left[  1/2+\sqrt{1/4+\log_{2}\left(  \left(  j+1\right)  /2\right)
},1/2+\sqrt{1/4+\log_{2}j}\right]  \cap\mathbb{N}.
\]
The sequence $(x^{(j)})_{j=1}^{\infty}$ defined by
\[
x^{(j)}={2}^{{\frac{( 1-k_{j})( k_{j}(2+p)+2p) }{2p}}}\cdot\sum_{s=1}%
^{2^{k_{j}( k_{j}-1) }}\left(  a_{rs} ^{(2^{k_{j}(k_{j}-1)})} \cdot
e_{2^{k_{j}( k_{j}-1) }+s-1}\right)
\]
if $A_{j}\neq\emptyset$, where $r=j-2^{k_{j}(k_{j}-1)}+1$ and $k_{j}$ is the
unique element of $A_{j}$, and $x^{(j)}=0$ otherwise, is unconditionally
summable in $\ell_{p}$ and
\[%
%TCIMACRO{\tsum \limits_{j=1}^{\infty} }%
%BeginExpansion
{\textstyle\sum\limits_{j=1}^{\infty}}
%EndExpansion
\left\Vert x^{(j)}\right\Vert ^{2-\varepsilon}=\infty
\]
for all $\varepsilon>0.$
\end{theorem}

\begin{proof}
Recall that $g_{1}=f_{0}$, $g_{2}=f_{1}$, $g_{3}=f_{2}^{(1)}$, $g_{4}%
=f_{2}^{(2)}$, and so on. For unit vector $\varphi\in\ell_{p^{\ast}}^{n}$, we
have%
\[
\sum\limits_{i=1}^{n}\left\vert \sum\limits_{j=1}^{n}\varphi_{j}a_{i,j}%
^{(n)}\right\vert \leq n^{\frac{1}{2}+\frac{1}{p}}%
\]
where $\varphi=(\varphi_{1},\ldots,\varphi_{n})$. In fact, for each
$i=1,2,\ldots,n$, by the definition of $g_{j}$ in the interval $((i-1)/n,i/n)$
we obtain
\[
\int_{(i-1)/n}^{i/n}\left\vert \sum_{j=1}^{n}\varphi_{j}g_{j}(t)\right\vert
dt=\frac{1}{n}\left\vert \sum\limits_{j=1}^{n}\varphi_{j}a_{i,j}%
^{(n)}\right\vert ,
\]
i.e.,
\begin{equation}
\left\vert \sum\limits_{j=1}^{n}\varphi_{j}a_{i,j}^{(n)}\right\vert
=n\int_{(i-1)/n}^{i/n}\left\vert \sum_{j=1}^{n}\varphi_{j}g_{j}(t)\right\vert
dt, \label{i1}%
\end{equation}
Hence, by (\ref{i1})
\begin{equation}
\sum\limits_{i=1}^{n}\left\vert \sum\limits_{j=1}^{n}\varphi_{j}a_{i,j}%
^{(n)}\right\vert =n\sum_{i=1}^{n}\int_{(i-1)/n}^{i/n}\left\vert \sum
_{j=1}^{n}\varphi_{j}g_{j}(t)\right\vert dt=n\int_{0}^{1}\left\vert \sum
_{j=1}^{n}\varphi_{j}g_{j}(t)\right\vert dt. \label{i2}%
\end{equation}
By (\ref{i2}) and by the monotonicity of the $L_{p}$-norms, we have
\begin{equation}
\sum\limits_{i=1}^{n}\left\vert \sum\limits_{j=1}^{n}\varphi_{j}a_{i,j}%
^{(n)}\right\vert =n\int_{0}^{1}\left\vert \sum_{j=1}^{n}\varphi_{j}%
g_{j}(t)\right\vert dt\leq n\left(  \int_{0}^{1}\left\vert \sum_{j=1}%
^{n}\varphi_{j}g_{j}(t)\right\vert ^{2}dt\right)  ^{\frac{1}{2}}. \label{i3}%
\end{equation}
Since $\int_{0}^{1}g_{j}(t)g_{k}(t)dt=\delta_{jk}$, for $j,k=1,\ldots,n$, we
have
\begin{equation}
\int_{0}^{1}\left\vert \sum_{j=1}^{n}\varphi_{j}g_{j}(t)\right\vert
^{2}dt=\sum_{j=1}^{n}|\varphi_{j}|^{2}. \label{i4}%
\end{equation}
Then, by (\ref{i3}) and (\ref{i4}), we have
\begin{equation}
\sum\limits_{i=1}^{n}\left\vert \sum\limits_{j=1}^{n}\varphi_{j}a_{i,j}%
^{(n)}\right\vert \leq n\left(  \sum_{j=1}^{n}|\varphi_{j}|^{2}\right)
^{\frac{1}{2}}. \label{i5}%
\end{equation}
On the other hand, by the H\"{o}lder inequality, using that $\left\Vert
\varphi\right\Vert _{p^{\ast}}\leq1$, we obtain
\begin{equation}
\left(  \sum_{j=1}^{n}|\varphi_{j}|^{2}\right)  ^{\frac{1}{2}}\leq\left(
\sum_{j=1}^{n}1\right)  ^{\frac{1}{2}-\frac{1}{p^{\ast}}}\cdot\left(
\sum_{j=1}^{n}|\varphi_{j}|^{p^{\ast}}\right)  ^{\frac{1}{p^{\ast}}}\leq
{n}^{\left(  \frac{1}{2}-\frac{1}{p^{\ast}}\right)  }. \label{i6}%
\end{equation}
Therefore, by (\ref{i5}) and (\ref{i6}), we have
\[
\sup_{\left\Vert \varphi\right\Vert _{p^{\ast}}\leq1}\sum\limits_{i=1}%
^{n}\left\vert \sum\limits_{j=1}^{n}\varphi_{j}a_{i,j}^{(n)}\right\vert \leq
n\left(  \sum_{j=1}^{n}|\varphi_{j}|^{2}\right)  ^{\frac{1}{2}}\leq n\cdot
{n}^{\left(  \frac{1}{2}-\frac{1}{p^{\ast}}\right)  }=n^{\frac{1}{2}+\frac
{1}{p}}.
\]
The rest of the proof follows the lines of the proof of Theorem \ref{casecomp}%
. It suffices to consider $\alpha=2$ in the definition of $j_{k}$ and change
the elements of the complex matrix by the elements of the Walsh system.
\end{proof}

\end{document}